\documentclass[preprint,review,10pt]{amsart}
\usepackage{amsmath,amssymb,amsfonts,xspace,mathrsfs}
\newtheorem{theorem}{Theorem}[section]

\theoremstyle{definition}
\newtheorem{definition}[theorem]{Definition}
\newtheorem{example}[theorem]{Example}

\newtheorem{proposition}[theorem]{Proposition}
\newtheorem{corollary}[theorem]{Corollary}

\theoremstyle{remark}

\numberwithin{equation}{section}

\begin{document}

\title{weakly $Q$-ideals of commutative rings}

\author{Mahdi Anbarloei}
\address{Department of Mathematics, Faculty of Sciences,
Imam Khomeini International University, Qazvin, Iran.
}

\email{m.anbarloei@sci.ikiu.ac.ir}

%\author{}
%\address{}
%\email{}

%    General info

%\date{September  , 2013.}

%------------------------------------------------------------------------------
%%%%%%%%%%%%%%%%%%%%%%%%%%%%%%%%%%%%%%%%%%%%%%%%%%%%%%%%%%%%%%%%%%%%%%%%%%%%%%%%%%%%%%%%%%%%%%%%%%%%%%

%%%%%%%%%%%%%%%%%%%%%%%%%%%%%%%%%%%%%%%%%%%%%%%%%%%%%%%%%%%%%%%%%%%%%%%%%%%%%%%%%%%%%%%%%%%%%%%%%%%%%%%%%%%%%%%%%%%%%%%%%%%%%%%%%%%%%%%%%
\begin{abstract}
 In this paper, we introduce and study the notion of  weakly $Q$-ideals in  commutative rings. 

\end{abstract}
%%%%%%%%%%%%%%%%%%%%%%%%%%%%%%%%%%%%%%%%%%%%%%%%%%%%%%%%%%%%%%%%%%%%%%%%%%%%%%%%%
\maketitle

\section{Introduction} 
Throughout this paper, we assume  that all rings are commutative with identity. Let $A$ be a commutative ring and $I$ be an ideal of $A$. We will denote by $Z(A)$, $\text{U}(A)$ and $\text{Max}(A)$, the set of zero-divisor elements, unit elements and maximal ideals of $A$, respectively. Moreover, $\sqrt{I}$ denotes the radical of $I$. If $\sqrt{I}=I$, then $I$ is called a radical ideal.

The development of various classifications of ideals has provided mathematicians with deeper insights into the fundamental nature of ring structures. The important role of prime ideals within commutative ring theory has motivated a wealth of generalizations and diverse research trajectories. The significance of some generalizations is on par with  prime ideals. A Seminal turning point in this direction was established by Anderson and  Smith, who introduced the notion of weakly prime ideals, paving the way for a more profound analysis of structural properties where the product of elements vanishes.   This paradigm shift demonstrated  that distinguishing between   zero products and  non-zero expressions can unveil rich structural features of  rings. A proper $I$ of $A$ is a weakly prime ideal if $0 \neq xy \in I$ for $x,y \in A$, then $x \in I$ or $y \in I$ \cite{Anderson2}. Later, Atani and Farzalipour presented the idea of weakly primary ideals, that is a proper ideal $I$ of $A$ such that for all $x,y \in A$ with $0 \neq xy \in I$, $x \in I$ or $y \in \sqrt{I}$ \cite{Atani}. Khashan and Celikel introduced the concept of  weakly $J$-ideals via the Jacobson radical  of $A$, denoted by $J(A)$ in \cite{Khashan}. They said that a proper ideal $I$ of $A$ is a weakly $J$-ideal  if whenever $x,y \in A$ such that $0 \neq xy \in I$ and $x \notin J(A)$, then $y \in I$. In 2025, by utilizing  the nilradical  of $A$   denoted by $\text{N}(A)$, Ersoy et al. proposed the notion of weakly $n$-ideals. A proper ideal $I$ of $A$ is a weakly $n$-ideal if whenever $0 \neq xy \in I$ for $x,y \in A$,  then either $x \in I$ or $y \in \text{N}(A)$ \cite{Ersoy}.

The notion of $Q$-ideals was introduced by Mimouni, utilizing a fixed ideal $Q$ of $A$.  \cite{Mimouni}. Call a proper ideal $I$ of $A$ a $Q$-ideal if $xy \in I$ where $x,y \in A$ and $x \notin Q$ imply $y \in I$. Our objective in this paper is to introduce and investigate  the notions of weakly $Q$-ideals, which represents a   attempt to merging the concepts  of weakly prime  ideals , weakly primary ideals, weakly $J$-ideals and weakly $n$-ideals in one framework.

%We investigate how this concept behaves under standard ring-theoretic constructions, including localization via multiplicative subsets, the formation of direct products, and the mechanism of Nagata’s idealization. Furthermore, we demonstrate the utility of this notion by establishing a new variant of Nakayama’s lemma and exploring its interplay with local domains. 

 %The primary objective of this paper is to introduce and rigorously investigate the concept of weakly $Q$-ideals, providing a comprehensive framework for their characterization in commutative rings

%We explore how the condition of being a weakly $Q$-ideal interacts with classical ring-theoretic constructions, including direct products and localization

%One of the significant contributions of this study is the development of a structural equivalence between weakly $Q$-ideals of a ring $A$ and their counterparts in the idealization $A(+)M$

%We establish a series of equivalent characterizations for weakly $Q$-ideals, utilizing quotient ideal properties and internal multiplication conditions

%%%%%%%%%%%%%%%%%%%%%%
\section{weakly $Q$-ideals}
\begin{definition}
Let $I$ be a proper ideal of a commutative ring $A$. We say that $I$ is a weakly $Q$-ideal of $A$ if whenever $x,y \in A$ such that $0 \neq xy \in I$ and $x \notin Q$, then $y \in I$.
\end{definition}
While every $Q$-ideal is a weakly $Q$-ideal, the following example shows that the converse is not true in general.
\begin{example}
Let $A=\mathbb{Z}_{30}$ and $Q=\langle 15 \rangle$. Then, the ideal $I=\langle 0 \rangle$ is a weakly $Q$-ideal of $A$, but it is not $Q$-ideal. Indeed, $6 \cdot 5 \in I$ and $6,5 \notin Q$ but clearly $5,6 \notin I$.
\end{example}
\begin{proposition} \label{1}
Let $I$ and $Q$ be  ideals of $A$. If $I$ is a weakly $Q$-ideal of $A$, then $I \subseteq Q$.
\end{proposition}
\begin{proof}
Assume that $I$ is a weakly $Q$-ideal of $A$ such that  $I \nsubseteq Q$. Then, there exists $x \in I$ such that $x \notin Q$. Since $I$ is a weakly $Q$-ideal of $A$, $x \cdot 1 \in I$ and $x \notin Q$, we get $1 \in I$, a contradiction. Thus, $I \subseteq Q$, as required.
\end{proof}
\begin{theorem}\label{2}
Let $I$ be an ideal of $A$. Then, $I$ is a weakly $J$-ideal of $A$ if and only if $I$ is a weakly $M$-ideal for every  $M \in \text{Max}(A)$.
\end{theorem}

\begin{proof}
Let $I$ be a weakly $J$-ideal of $A$ and $M \in \text{Max}(A)$. Assume that $0 \neq xy \in I$ and $x \notin M$. This implies that $x \notin \text{J}(A)$. By the hypothesis, we conclude that $y \in I$. Thus, $I$ is a weakly $M$-ideal of $A$. Conversely, let $I$ be a weakly $M$-ideal for every  $M \in \text{Max}(A)$. By Proposition \ref{1}, $I \subseteq M$ for every $M \in \text{Max}(A)$ which means $\text{J}(A)$ contains $I$. Now, assume that $0 \neq xy \in I$ and $x \notin \text{J}(A)$. This implies that $x \notin M$ for some $M \in \text{Max}(A)$. Hence, $y \in I$ as $I$ is a weakly $M$-ideal of $A$. Consequently, $I$ is a weakly $J$-ideal of $A$.
\end{proof}
\begin{theorem}\label{3}
Let $A$ be a domain.  Then, the following statements are equivalent:
\begin{enumerate}
\item $A$ is a  local ring.
\item  Every  proper ideal $I$ of $A$  is a weakly  $M$-ideal for every maximal ideal $M$ containing $I$.
\item  Every proper principle ideal $I$ of $A$  is a weakly  $M$-ideal for every maximal ideal $M$ containing $I$.
\end{enumerate}    
\end{theorem}
\begin{proof}
(1) $\Longrightarrow$ (2)   Let   $I$ be a proper ideal of $A$ and $M$ be the unique maximal ideal of $A$. Assume that  $0 \neq xy \in I$ for $x,y \in A$ and $x \notin M$. This implies that $x \in \text{U}(A)$ which means $x^{-1}(xy) \in I$. Thus, $I$ is a weakly  $M$-ideal of $A$.

(2) $\Longrightarrow$ (3) Clear.

(3) $\Longrightarrow$ (1) Let every proper principle ideal $I$ of $A$  be a weakly  $M$-ideal for every maximal ideal $M$ containing $I$. Assume that $M \in \text{Max}(A)$. We show that every $x \notin M$ is a unit. Take any $x \in A \backslash M$ and $0 \neq y \in M$. Since $A$ is a domain, we have $xy \neq 0$. Put $I=\langle xy \rangle$. Since $I$ is a weakly $M$-ideal of $A$, $0 \neq xy \in I$ and $x \notin M$, we get $y \in I=\langle xy \rangle$. Then, $y=rxy$ and so $rx=1$ as $A$ is a domain. This means $x \in \text{U}(A)$. Consequently, $A$ is local.
\end{proof}
\begin{theorem}\label{4}
Let $I$ and $Q$ be  ideals  of $A$. Then, the following statements are equivalent:
\begin{enumerate}
\item $I$ is a weakly $Q$-ideal of $A$.
\item $(I:x)=I \cup (0:x)$ for each $x \in A \backslash  Q$.
\item $(I:x) \subseteq Q \cup (0:x)$ for each $x \in A \backslash I$.
\item If $x \in A$ and $I_1$ is an ideal of $A$ with $0 \neq I_1x\subseteq I$, then $I_1 \subseteq Q$ or $x \in I$.
\item If $I_1$ and $I_2$ are ideals of $A$ with $0 \neq I_1I_2 \subseteq I$, then $I_1 \subseteq Q$ or $I_2 \subseteq I$.
\end{enumerate}
\end{theorem}
\begin{proof}
(1) $\Longrightarrow $ (2) Assume that $x \in A \backslash Q$. The inclusion $I \cup (0 : x) \subseteq (I:x)$ holds. Let $y \in (I : x)$. Then, we have $xy \in I$. If $xy \neq 0$, then we get $y \in I$ as $I$ is a weakly $Q$-ideal of $A$. If $xy=0$, then we obtain $y \in (0:x)$ which means $(I:x) \subseteq I \cup (0:x)$. Thus, $(I:x)=I \cup (0:x)$.

(2) $\Longrightarrow $ (1) Assume that $0 \neq xy \in I$ for $x,y \in A$ and $x \notin Q$. By the hypothesis, $y \in I \cup (0 : x)$. Since $y \notin (0:x)$, we have $y \in I$. 

(1) $\Longrightarrow $ (3) By a similar argument to (1) $\Longrightarrow $ (2), we conclude that  the claim is true.

(3) $\Longrightarrow $ (4) Let $0 \neq I_1 x \subseteq I$    for some ideal $I_1$ of $A$  such that $x \in A \backslash I$. Since  $(I : x) \neq (0 : x)$, we obtain  $(I : x) \subseteq Q$ by the hypothesis. Hence, $I_1 \subseteq Q$.

(4) $\Longrightarrow $ (5) Let $0 \neq I_1I_2 \subseteq I$ for some ideals $I_1$ and $I_2$  of $A$ such that neither $I_1 \subseteq Q$ nor $I_2 \subseteq I$. From $0 \neq I_1I_2 \subseteq I$ it follows that $0 \neq I_1 x \subseteq I$ for some $x \in I_2$. This implies that $x \in I$ as $I_1 \nsubseteq Q$. Take any $y \in I_2 \backslash I$. If $I_1 y \neq 0$, then we obtain $y \in I$, a contradiction. Therefore, $I_1 y=0$. Since $0 \neq I_1(x+y) \subseteq I$ and $I_1 \nsubseteq Q$, we obtain $x+y \in I$ and so $y \in I$, a contradiction. 

(5) $\Longrightarrow $ (1) Let $0 \neq xy \in I$ for $x,y \in A$. Set $I_1=\langle x \rangle$ and $I_2=\langle y \rangle$. Then, $0 \neq I_1 I_2 \subseteq I$. By the hypothesis, we conclude that $x \in I_1 \subseteq Q$ or $y \in I_2 \subseteq I$. 
\end{proof}
\begin{proposition}
Let $A$ is a ring. 
\begin{enumerate}
\item If $I$ is a weakly $Q$-ideal   of $A$ for every ideal $Q$ of $A$ properly containing $I$  such that $I$ is not radical, then $I$ is a weakly primary ideal.
\item  If $I$ is a weakly $Q$-ideal of $A$ for every ideal $Q$ of $A$ properly containing $\sqrt{I}$ and $ 0$ is a quasi prime ideal of $A$, then $\sqrt{I}$ is a weakly prime ideal.
\item The following statements are equivalent:
\begin{enumerate}
\item[(i)] $I$ is a weakly primary ideal.
\item[(ii)] $I$ is a weakly $\sqrt{I}$-ideal.
\item[(iii)] $I$ is a weakly $Q$-ideal of $A$ for every ideal $Q$ of $A$ containing $\sqrt{I}$.
\end{enumerate}
\end{enumerate}
\begin{proof}
(1) Assume that  $I$ is not radical. By the assumption, $I$ is a weakly $\sqrt{I}$-ideal of $A$. Then, $(I:x)=I \cup (0:x)$ for each $x \notin A \backslash \sqrt{I}$ by Theorem \ref{4}. Therefore, we conclude that $I$ is   a weakly primary ideal by Proposition 2.1 in \cite{Atani}.

(2) Let $I$ be  a weakly $Q$-ideal of $A$ for every ideal $Q$ of $A$ properly containing $\sqrt{I}$. Assume that $0 \neq xy \in \sqrt{I}$ and $x \notin \sqrt{I}$. This implies that $x^ny^n \in I$ for some $n \in \mathbb{N}$. Let $x^ny^n=0$. Since $0$ is a quasi prime ideal of $A$ and $x^n \notin \sqrt{0}$, we get $y^n \in \sqrt{0}$. This means that $y^{nm}=0$ for some $m \in \mathbb{N}$. Hence, we have $y \in \sqrt{0} \subseteq \sqrt{I}$ which implies  $\sqrt{I}$ is a weakly prime ideal of $A$. If   $x^ny^n \neq 0$, then one can easily complete the proof  by using an argument similar to that in the proof of Theorem 2.3(4)  in \cite{Mimouni}.

(3) (i) $\Longrightarrow$ (ii) and (ii) $\Longrightarrow$ (iii) are obvious.

(iii) $\Longrightarrow$ (i) By the hypothesis, $I$ is a weakly $\sqrt{I}$-ideal. Assume that  $0 \neq xy \in I$ and $x \notin I$. If $y \notin \sqrt{I}$, then $x \in I$, a contradiction. Then,   $y \in \sqrt{I}$  which implies  $I$ is a weakly primary ideal.
\end{proof}
\end{proposition}
\begin{proposition}
Let $Q$ be an   ideal of $A$ and $\{I_i\}_{i \in \Delta}$ be a non-empty family of weakly $Q$-ideals of $A$. Then, $\cap_{i \in \Delta}I_i$ is a weakly $Q$-ideal. 
\end{proposition}
\begin{proof}
Let $0 \neq xy \in \cap_{i \in \Delta}I_i$ for $x,y \in A$ and $x \notin Q$. Since $ I_i $ is a  weakly $Q$-ideal of $A$ for all $i \in \Delta$, we get $y \in \cap_{i \in \Delta}I_i$. Thus, $\cap_{i \in \Delta}I_i$ is a weakly $Q$-ideal. 
\end{proof}
\begin{proposition}
Let  $I$ and $Q$ be ideals of $A$ and let $B$ be a non-empty subset of $A$ such that $B \nsubseteq I$. If $(0:B)$  and  $I$ are weakly $Q$-ideals of $A$, then $(I:B)$ is a weakly $Q$-ideal.  
\end{proposition}
\begin{proof}
Let $I$  be an ideal  of $A$ and   $B$ be a non-empty subset of $A$ where $B \nsubseteq I$. Clearly, $(I:B)$ is a proper ideal of $A$. Assume that $0 \neq xy \in (I:B)$ for $x,y \in A$ and $x \notin Q$. So, we have $xyB \subseteq I$. Let  $xyB=0$. Since  $(0:B)$  is a weakly $Q$-ideal  of $A$ and $0 \neq xy \in (0:B)$, we get $y \in (0:B)$ and so $y \in (I:B)$. Let $xyB \neq 0$.    By Theorem \ref{4}, we obtain $yB \subseteq   I$ as $I$ is a  weakly $Q$-ideals of $A$. This implies that $y \in (I:B)$. Hence, $(I:B)$ is a weakly $Q$-ideal.
\end{proof}
\begin{theorem} \label{5}
Let $I$ and $Q$ be  ideals  of $A$.  If $I$ is a weakly $Q$-ideal of $A$ that is not $Q$-ideal, then $I^2=0$.
\end{theorem}
\begin{proof}
Assume that $I^2 \neq 0$. We show that $I$ is a $Q$-ideal. Let $xy \in I$ and $x \notin Q$. If $xy \neq 0$, then we get $y \in I$ as $I$ is a weakly $Q$-ideal of $A$. Let $xy =0$. If $xI \neq 0$, then there exists $a \in I$ such that $xa \neq 0$. Then, we have $0 \neq x(y+a) \in I$. This implies that $y+a \in I$ as $I$ is a weakly $Q$-ideal of $A$ and $x \notin Q$. So, we get $y \in I$. If $yI \neq 0$, then there exists $b \in I \subseteq Q$ such that $yb \neq 0$. Since $0 \neq yb =y(b+x)  \in I$ and  $b+x \notin Q$, we get $y \in I$. Let us assume that $Ix=Iy=0$. From $I^2 \neq 0$ it follows that $ab \neq 0$ for some $a,b \in I$. Since $0 \neq ab=(a+y)(b+x) \in I$ and $b+x \notin Q$, we get $a+y \in I$ and so $y \in I$. Consequently, $I$ is a $Q$-ideal of $A$.
\end{proof}
Now, we state a version of Nakayama$^,$s lemma.

\begin{corollary} \label{6}
Let   $I$ be a weakly $Q$-ideal of $A$ that is not $Q$-ideal. If $M$ is an $R$-module and $IM=M$, then $M=0$.
\end{corollary}
\begin{proof}
Let $I$ is  a weakly $Q$-ideal of $A$ that is not $Q$-ideal. By  Theorem \ref{5}, we have $I^2=0$. Assume that  $IM=M$ for some $R$-module $M$. Then, we have $M=IM=I^2M=0$.
\end{proof}
Given   a proper ideal $I$ of $A$, we use $Z_I(A)$ to denote  the set $\{a \in A \mid as \in I \  \text{for some} \ s \in A \backslash I\}$.
\begin{theorem} \label{7}
Let $I$ and $Q$ be ideals of $A$ and $S$ be a multiplicative subset of $A$ with $Q \cap S=\varnothing$. 
\begin{enumerate}
\item If $I$ is a weakly $Q$-ideal of $A$, then $S^{-1}I$ is a weakly $S^{-1}Q$-ideal of $S^{-1}A$.
\item If $S^{-1}I$ is a weakly $S^{-1}Q$-ideal of $S^{-1}A$ and $S \cap Z(A)=S \cap Z_i=\varnothing$ for $i \in\{I,Q\}$, then $I$ is a weakly $Q$-ideal of $A$.
\end{enumerate} 
\end{theorem}
\begin{proof}
(1) Let $I$ be a weakly $Q$-ideal of $A$. Assume that $0 \neq \frac{x}{s_1}\frac{y}{s_2} \in S^{-1}I$ for  $\frac{x}{s_1},\frac{y}{s_2} \in S^{-1}A$ such that $\frac{x}{s_1} \notin Q$. Therefore, we have $0 \neq txy \in I$ for some $t \in S$. Since $I$ is a weakly $Q$-ideal of $A$ and $x \notin Q$, we get $ty \in I$. Then, $\frac{y}{s_2}=\frac{ty}{ts_2} \in S^{-1}I$. Consequently,  $S^{-1}I$ is a weakly $S^{-1}Q$-ideal of $S^{-1}A$.

(2) Let  $S^{-1}I$ be a weakly $S^{-1}Q$-ideal of $S^{-1}A$. Assume that $0 \neq xy \in I$ for $x,y \in A$ and $x \notin Q$. This implies that $\frac{x}{1}\frac{y}{1} \in S^{-1}I$. If $\frac{x}{1} \in S^{-1}Q$, then $rx \in Q$ for some $r \in S$. Since  $S \cap Z_Q(A)=\varnothing$, we have  $x \in Q$ and so $\frac{x}{1} \in S^{-1}Q$, a contradiction. Let $\frac{x}{1}\frac{y}{1}=0$. Then, there exists $t \in S$ such that $txy=0$. This means that $xy=0$ as $S \cap Z(A)=\varnothing$ which is a contradiction. Then, $0 \neq \frac{x}{1}\frac{y}{1} \in S^{-1}I$.   Since $S^{-1}I$ is a weakly $S^{-1}Q$-ideal of $S^{-1}A$ and  $\frac{x}{1} \notin  S^{-1}Q$, we conclude that  $\frac{y}{1} \in S^{-1}I$. Hence,  $uy \in I$ for some $u \in S$. From $S\cap Z_I(A)=\varnothing$ it follows that $y \in I$. Thus, $I$ is a weakly $Q$-ideal of $A$.
\end{proof}
\begin{theorem} \label{8}
Let $A_1$ and $A_2$ be two rings and let $I_i$ and $Q_i$ be ideals of $A_i$ for each $i \in \{1,2\}$ such that $I_1 \times I_2$ be a non-zero proper ideal of $A_1 \times A_2$. Then, the following statements are equivalent:
\begin{enumerate}
\item $I_1 \times I_2$ is a weakly $Q_1 \times Q_2$-ideal of $A_1 \times A_2$.
\item $I_1$ is a   $Q_1$-ideal of $A_1$ and $I_2=A_2$ or  $I_2$ is a   $Q_2$-ideal of $A_2$ and $I_1=A_1$.
\item $I_1 \times I_2$ is a   $Q_1 \times Q_2$-ideal of $A_1 \times A_2$.
\end{enumerate}
\end{theorem}
\begin{proof}
(1) $\Longrightarrow$ (2)  Let $I_1$ and $I_2$ be proper ideals of $A_1$ and $A_2$, respectively. Assume that $0 \neq (x,y) \in I_1 \times I_2$. This means that $0 \neq (x,1)(1,y) \in I_1 \times I_2$ such that $(x,1) \notin Q_1 \times Q_2$ and $(1,y) \notin I_1 \times I_2$. This contradicts the fact that $I_1 \times I_2$ is a   weakly $Q_1 \times Q_2$-ideal of $A_1 \times A_2$. Let us assume that $I_1$ is a proper ideal of $A_1$ and $I_2=A_2$. By corollary  \ref{6}, $I_1 \times I_2$ is a $Q_1\times Q_2$-ideal of $A_1 \times A_2$ as $(I_1 \times I_2)^2 \neq 0$. Let $xy \in I_1$ for $x,y \in A_1$ and $x \notin Q_1$. Since $(x,0)(y,0) \in I_1 \times I_2$ and $(x,0) \notin Q_1 \times Q_2$, we have $(y,0) \in I_1 \times I_2$ which implies $y \in I_1$. 

(2) $\Longrightarrow$ (3) Let us assume that $I_1$ is a   $Q_1$-ideal of $A_1$ and $I_2=A_2$. Let $(x,y)(w,z) \in I_1 \times I_2$ for $(x,y),(w,z) \in I_1 \times A_2$ and $(x,y) \notin Q_1 \times Q_2$. From $xw \in I_1$ it follows that $w \in I_1$ as $x \notin Q_1$. Therefore, we conclude that $(w,z) \in   I_1 \times A_1$. Thus, $I_1 \times I_2$ is a   $Q_1 \times Q_2$-ideal of $A_1 \times A_2$.

(3) $\Longrightarrow$ (1) It is obvious.
\end{proof}
Let   $M$ be an $A$-module where $A$ is  a commutative ring with identity $1$. Recall from \cite{Anderson} that   $A(+)M=\{(a,m) \mid a \in A, m \in M\}$ with  the  addition  $ (a, m) + (b, n) = (a + b, m + n) $ and the multiplication   $(a, m)(b, n) = (ab, an + bm)$  is a  ring with identity $(1, 0)$.  This commutative ring is called the idealization of  $M$. Assume that  $I$ is an ideal of $A$ and $N$ a submodule of $M$. Theorem 3.3 in \cite{Anderson} shows that  $I(+)N$ is an ideal of $A(+)M$ if and only if $IM \subseteq N$.

 Next, we determine the relation between weakly $Q$-ideals of $A$ and weakly $Q(+)M$-ideals of $A(+)M$. 
 \begin{theorem} \label{9}
 Let $I$ and $Q$ be proper  ideals of $A$ and let $N$ be a submodule of an $A$-module $M$. If $I(+)N$ is a weakly $Q(+)M$-ideal of  $A(+)M$, then $I$ is a weakly $Q$-ideal of $A$. 
 \end{theorem}
 \begin{proof}
 Assume that $0 \neq xy \in I$ for $x,y \in A$ such that $x \notin Q$. This implies that $(0,0) \neq (x,0)(y,0) \in I(+) N$ but $(x,0) \notin Q(+)M$. Since $I(+)N$ is a weakly $Q(+)M$-ideal of  $A(+)M$, we get $(y,0) \in I(+)N$ and so $y \in I$. Hence, $I$ is a weakly $Q$-ideal of $A$.
 \end{proof}
 %The following example verifies that the converse of Theorem \ref{9} may not be true. 
%  \begin{example}
  %The ideal $0(+)\langle  6 \rangle$ is not a weakly $0(+)\langle  3 \rangle$-ideal of $\mathbb{Z}(+)\mathbb{Z}$ since $(0,0) \neq (3,3)(0,2) \in 0(+)\langle  6 \rangle$  and $(3,3),(0,2) \notin 0(+)\langle  3 \rangle$ but $(0,2),(3,3) \notin 0(+)\langle  6 \rangle$. However, $0$ is a weakly $
  %\end{example}
  \begin{theorem} \label{10}
 Let $I$ and $Q$ be proper  ideals of $A$ and let $N$ be a submodule of an $A$-module $M$. Then, the following statements are equivalent:
\begin{enumerate}
\item $I(+)M$ is a weakly $Q(+)M$-ideal of $R(+)M$.
\item   $I$ is a weakly $Q$-ideal of $A$ and if $a, b \in A$ with $ab = 0$ but $a \notin Q$ and $b \notin I$, then $a, b \in \operatorname{Ann}(M)$.
 \end{enumerate}
 \end{theorem}
 \begin{proof}
 (1) $\Longrightarrow$ (2) Assume that $I(+)M$ is a weakly $Q(+)M$-ideal of $R(+)M$. By Theorem \ref{9}, $I$ is a a weakly $Q$-ideal of $A$. Let $ab=0$ for $a,b \in A$ such that neither $a \in Q$ nor $b \in I$. If $a \notin Ann(M)$, then $am \neq 0$ for some $m \in M$. Then, we have $(0,0) \neq (a,0)(b,m) \in I (+) M$ such that $(a,0) \notin Q(+)M$ and $(b,m) \notin I(+)M$, a contradiction. Hence, $a \in Ann(M)$. By a similar argument, we conclude that $b \in Ann(M)$.
 
  (2) $\Longrightarrow$ (1) Let $(0,0) \neq (x,m_1)(y,m_2) \in I(+)M$ for $ (x,m_1),(y,m_2) \in  A(+)M$ but $(x,m_1) \notin Q(+)M$. This implies that $  xy \in I$ and $x \notin Q$. Let $xy=0$ such that $x \notin Q$ and $y \notin I$. Therefore, we have $x,y \in Ann(M)$ by the hypothesis which means $(x,m_1)(y,m_2)=(0,0)$ which is impossible. Then, $xy \neq 0$. Since $I$ is a $Q$-ideal of $A$ and $x \notin Q$, we have $y \in I$ which implies $(y,m_2) \in I(+)M$. Thus, $I(+)M$ is a weakly $Q(+)M$-ideal of $R(+)M$.
 \end{proof}
%%%%%%%%%%%%%%%%%%%%%%%%%%%%%%%%%%%%%%%%%%%
%%%%%%%%%%%%%%%%%%%%%%%%%

\end{document}